\documentclass[12pt]{amsart}
\usepackage{amsmath}
\usepackage{amssymb}
\usepackage{cite}

\newtheorem{theorem}{Theorem}[section]
\newtheorem{lemma}[theorem]{Lemma}
\newtheorem{proposition}[theorem]{Proposition}

 \theoremstyle{definition}
\newtheorem{definition}[theorem]{Definition}

\theoremstyle{remark}
\newtheorem{remark}[theorem]{Remark}

\numberwithin{equation}{section}

\begin{document}

\title{Dirichlet problem for degenerate Hessian quotient type curvature equations}

\author{Xiaojuan Chen}
\address{Faculty of Mathematics and Statistics, Hubei Key Laboratory of Applied Mathematics, Hubei University,  Wuhan 430062, P.R. China}
\email{201911110410741@stu.hubu.edu.cn}

\author{Qiang Tu$^{\ast}$}
\address{Faculty of Mathematics and Statistics, Hubei Key Laboratory of Applied Mathematics, Hubei University,  Wuhan 430062, P.R. China}
\email{qiangtu@hubu.edu.cn}

\author{Ni Xiang}
\address{Faculty of Mathematics and Statistics, Hubei Key Laboratory of Applied Mathematics, Hubei University,  Wuhan 430062, P.R. China}
\email{nixiang@hubu.edu.cn}

\thanks{This research was supported by funds from the National Natural Science Foundation of China No. 11971157, 12101206; the Natural Science Foundation of Hubei Province, China, No. 2023AFB730.}
\thanks{$\ast$ Corresponding author}

\date{}

\begin{abstract}
In the paper, we prove the existence and uniqueness results of the $C^{1,1}$ regular graphic hypersurface for Dirichlet problem of a class of degenerate Hessian quotient type curvature equations under the condition $\psi^{\frac{1}{k-l}}\in C^{1,1}(\overline{\Omega}\times\mathbb{R}\times\mathbb{S}^n)$. Specially, we also consider the second order derivative estimates for the corresponding degenerate Hessian type curvature equations under the optimal condition $\psi^{\frac{1}{k-1}}\in C^{1,1}(\overline{\Omega}\times\mathbb{R}\times\mathbb{S}^n)$.
\end{abstract}



\maketitle
\vskip4ex

\section{Introduction}
In this paper, we consider Dirichlet problem for the following degenerate Hessian quotient type curvature equations
\begin{equation}\label{Eq1}
\frac{\sigma_k}{\sigma_l}(\lambda(\eta[M_u]))=\psi(X,\nu),\quad\mbox{in}~\Omega
\end{equation}
with homogenous boundary data. Here $\Omega\subset\mathbb{R}^{n}$ is a bounded domain, $M_u=\{(x,u(x));x\in\Omega\}$ is the graphic hypersurface defined by the function $u$, $X=(x,u(x))$ is the position vector of $M_u$, $\nu$ is the unit outward normal vector, $\sigma_k$ is the $k$-th elementary symmetric function, $\lambda(\eta[M_u])=(\lambda_1,\cdots,\lambda_n)$ is the eigenvalue vector of $g^{-1}\eta$ on $M_u$. $\psi\geq0$ is smooth enough with respect to every variable.
The $(0,2)$-tensor $\eta$ on $M_u$ is defined by
$$\eta_{ij}=Hg_{ij}-h_{ij},$$
where $g_{ij}$ and $h_{ij}$ are the first and second fundamental forms of $M_u$, respectively. $H$ is the mean curvature of $M_u$. Actually, $\lambda(\eta[M_u])$ in equation \eqref{Eq1} also arises in complex geometry, which has attracted the interest of many authors due to its geometric applications such as Gauduchon conjecture \cite{Ga84}. The admissible set for equation \eqref{Eq1} is defined as follows.
\begin{definition}
 A $C^2$ regular hypersurface $M_u\subset\mathbb{R}^{n+1}$ is called $(\eta,k)$-convex if its principal curvature vector $\kappa(X)\in\widetilde{\Gamma}_k$. As in \cite{CDH23}, $\widetilde{\Gamma}_k$ is defined by
 \begin{equation}\label{wid}
   \widetilde{\Gamma}_k:=\{\kappa=(\kappa_1,\cdots,\kappa_n)\in\mathbb{R}^n:\sigma_i(\lambda(\eta))>0, \lambda_i(\eta)=\sum_{j\neq i}\kappa_j, 1\leq i \leq k\}
 \end{equation}
for all $X\in M_u$. For $\Omega\subset\mathbb{R}^n$, a function $u: \Omega\rightarrow \mathbb{R}$ is called admissible if its graph is $(\eta,k)$-convex.
\end{definition}

If $\lambda(\eta[M_u])$ is replaced by the principal curvature $\kappa[M_u]$ and $l=0$, equation \eqref{Eq1} becomes the classical $k$-curvature equations
\begin{equation}\label{cur}
  \sigma_k(\kappa[M_u])=\psi(X,\nu).
\end{equation}
For the non-degenerate case, Dirichlet problem for equation \eqref{cur} has been achieved fruitful results. When $k=1,2$ and $n$, the left hand side of equation \eqref{cur} corresponds to the mean, scalar and Gauss curvature of the hypersurface $M_u$ respectively. The classical Dirichlet problem for the prescribed mean and Gauss curvature equations has been extensively studied, which can refer to \cite{Ba82, GT83, Se69, TU83}. For the intermediate case $1<k<n$, the first breakthroughs were due to Caffarelli-Nirenberg-Spruck \cite{CNS88} and Ivochkina \cite{IV1990, IV1991} for Dirichlet problem of equation \eqref{cur} with homogenous boundary data when $\psi$ is independent of $\nu$.
In fact, Ivochkina \cite{IV1990, IV1991} treated the case of $\sigma_k$. Caffaralli-Nirenberg-Spruck allowed for more general symmetric functions than $\sigma_k$, but still excluded Hessian quotient $\frac{\sigma_k}{\sigma_l}$. Then Lin-Trudinger \cite{LT94} considered Hessian quotient case $\frac{\sigma_k}{\sigma_l}$ with homogenous boundary date. Ivochkina-Lin-Trudinger \cite{ILT1996} extended  Lin-Trudinger's work to general boundary data but under exclusion of the case $k = n$. For the case $\frac{\sigma_n}{\sigma_l}$ with $0\leq l<n$ was subsequently studied by Ivochkina-Tomi \cite{IT98}. Recently, we considered Pogorelov type estimates of semi-convex solutions for Dirichlet problem of  equation \eqref{cur} with $(k+1)$-convex boundary data in \cite{CTX23}. The readers can refer to \cite{CNS86, GS04, GRW15, RW19, RW20, STW12, SUW04, SX17, Tru90, U00} for more researches about non-degenerate curvature equations.

The research of degenerate curvature equations can be tracked back to the degenerate Weyl problem studied by Guan-Li \cite{GL94} and Hong-Zuily \cite{HZ95}. Later on, Guan-Li \cite{GL97} studied the prescribed Gauss curvature problem. As we all know, for degenerate curvature equations, $C^{1,1}$ regularity of the solutions maybe the best regularity one can expect \cite{W95}. Guan-Zhang \cite{GZ} established $C^{1,1}$ estimates for a class of curvature type equations which is the combination of $\sigma_k$. When $\psi=\psi(x,u)\geq 0$ in equation \eqref{cur}, Jiao-Wang \cite{JW22} proved the existence of $C^{1,1}$ regular solutions for Dirichlet problem of equation \eqref{cur} with homogenous boundary data under the condition $\psi^{\frac{1}{k-1}}\in C^{1,1}(\mathbb{R}^n\times\mathbb{R})$. Then Jiao-Jiao \cite{Jiao22} established Pogorelov type estimates for Dirichlet problem of equation \eqref{cur} with $(k+1)$-convex boundary data under the condition $\psi^{\frac{1}{k-1}}\in C^{1,1}(\overline{\Omega}\times\mathbb{R})$.

Now, we review some of the facts on the non-degenerate case of equation \eqref{Eq1}. Chu-Jiao \cite{CJ21} studied the following prescribed curvature equations
\begin{equation}\label{xnu}
  \sigma_k(\lambda(\eta))=\psi(X,\nu),
\end{equation}
and they established curvature estimates for equation \eqref{xnu}. Then in \cite{CTX20}, we generalized Chu-Jiao's results to Hessian quotient case. It is of interest to consider the existence and regularity of solutions for curvature equations in graph form.
Specifically, Jiao-Sun \cite{JS22} considered Dirichlet problem for the following curvature equations
\begin{equation*}
  \det(\lambda(\eta[M_u]))=\psi(X,\nu),
\end{equation*}
and they derived the existence of  $C^{1,1}$ regular graphic $(\eta,n)$-convex hypersurfaces. Inspired by Jiao-Sun's work, it is natural to study Dirichlet problem of Hessian quotient type curvature equation \eqref{Eq1} for the degenerate case.

 In order to deal with equation \eqref{Eq1}, we usually need some geometric conditions on $\Omega$ as in \cite{IV1991}. If $k<n$, a bounded domain $\Omega$ in $\mathbb{R}^n$ is called uniformly $k$-convex if there exists a positive constant $K$ such that for each $x\in \partial\Omega$,
$$(\kappa_1^b(x),\cdots,\kappa_{n-1}^b(x),K)\in \Gamma_{k+1},$$
where $\kappa_1^b(x),\cdots,\kappa_{n-1}^b(x)$ are the principal curvatures of $\partial\Omega$ at $x$ and $\Gamma_k$ is the G{\aa}rding's cone defined by \eqref{gar}.
If $k=n$, it is easy to see that $n$-convexity is equivalent to strict convexity.

The main theorems are as follows.
\begin{theorem}\label{main}
Let $k\geq 2,~ 0\leq l<k<n$, $\Omega$ be a uniformly $k$-convex bounded domain in $\mathbb{R}^n$ with $\partial\Omega\in C^{2,1}$. Suppose that $\psi^{\frac{1}{k-l}}(x,z,\nu)\in C^{1,1}(\overline{\Omega}\times\mathbb{R}\times\mathbb{S}^n)\geq 0$ and $\psi_z\geq 0$. Assume that there exists an admissible subsolution $\underline{u}\in C^{1,1}(\overline{\Omega})$ satisfying $\kappa[M_{\underline{u}}]\in \widetilde{\Gamma}_k$ and
\begin{equation}\label{sub}
\left\{
\begin{aligned}
&\frac{\sigma_k}{\sigma_l}(\lambda(\eta[M_{\underline{u}}]))\geq \psi(\underline{X},\nu(\underline{X})), &&in~\Omega,\\
&\underline{u}= 0, &&on~\partial \Omega,
\end{aligned}
\right.
\end{equation}
where $\underline{X}=(x,\underline{u}(x))$ and $\nu(\underline{X})$ is the unit outward normal vector at $\underline{X}\in M_{\underline{u}}$. Then there exists a unique admissible solution $u\in C^{1,1}(\overline{\Omega})$ of Dirichlet problem
\begin{equation}\label{ali}
  \left\{
\begin{aligned}
&\frac{\sigma_k}{\sigma_l}(\lambda(\eta[M_u]))=\psi(X,\nu), &&in~\Omega,\\
&u = 0, &&on~\partial \Omega,
\end{aligned}
\right.
\end{equation}
satisfying
$$\|u\|_{C^{1,1}(\overline{\Omega})}\leq C,$$
where $C$ is a positive constant depending on $n, k, l, \Omega, \|\underline{u}\|_{C^{0,1}}$ and $\|\psi^{\frac{1}{k-l}}\|_{C^{1,1}}$.
\end{theorem}


\begin{remark}
The results of the classical degenerate curvature equations mainly focus on the case $\psi=\psi(X)$. Compared to the classical conclusions, we consider the more general right hand function $\psi=\psi(X,\nu)$.
\end{remark}

\begin{remark}
In the proof of the key Lemma 5.3 in Jiao-Sun \cite{JS22}, the special structure of the $(\eta,n)$-curvature plays an important role. Our challenge is how to achieve the similar result in the absence of that special structure.
\end{remark}



 In particular, we only need $\psi^{\frac{1}{k-1}}\in C^{1,1}(\overline{\Omega}\times\mathbb{R}\times\mathbb{S}^n)$ to establish second order derivative estimates for solutions of equation \eqref{Eq1} when $l=0$.


\begin{theorem}\label{2}
Let $\Omega$ be a uniformly $k$-convex bounded domain in $\mathbb{R}^n$ with $\partial\Omega\in C^{2,1}$. Suppose $2\leq k<n$, $\psi\geq 0$, $\psi^{\frac{1}{k-1}}\in C^{1,1}(\overline{\Omega}\times\mathbb{R}\times\mathbb{S}^n)$. Let $u\in C^4(\Omega)\cap C^2(\overline{\Omega})$ be an admissible solution of Dirichlet problem
\begin{equation}\label{Eq2}
  \left\{
\begin{aligned}
&\sigma_k(\lambda(\eta[M_u]))=\psi(X,\nu), &&in~\Omega,\\
&u = 0, &&on~\partial \Omega.
\end{aligned}
\right.
\end{equation}
Then there exists a positive constant $C$ depending on $n,k,\|u\|_{C^1}$ and $\|\psi^{\frac{1}{k-1}}\|_{C^{1,1}}$ satisfying
$$\sup_{\overline{\Omega}}|D^2u|\leq C.$$
\end{theorem}
In fact, the condition $\psi^{\frac{1}{k-1}}\in C^{1,1}(\overline{\Omega}\times\mathbb{R}\times\mathbb{S}^n)$ is optimal to derive the a priori estimates, which is consist with classical case based on Wang's counterexamples in \cite{W95}. But up to now, we can only derive $C^1$ estimates under the stronger condition $\psi^{\frac{1}{k}}\in C^{1,1}(\overline{\Omega}\times\mathbb{R}\times\mathbb{S}^n)$, it is still an open problem for Dirichlet problem \eqref{Eq2} to derive $C^1$ estimates under the condition $\psi^{\frac{1}{k-1}}\in C^{1,1}(\overline{\Omega}\times\mathbb{R}\times\mathbb{S}^n)$.

The organization of the paper is as follows. Some preliminaries are given in Section 2. We obtain $C^1$ estimates in Section 3. In Section 4, we establish the global second order derivative estimates. In Section 5, we deal with the boundary second order derivative estimates and complete the proof of Theorem \ref{main}, Theorem \ref{2}.
\section{Preliminaries}
\subsection{$k$-th elementary symmetric functions}
Let $\lambda=(\lambda_1,\dots,\lambda_n)\in\mathbb{R}^n$, then we recall
the definition of elementary symmetric function for $1\leq k\leq n$
\begin{equation*}
\sigma_k(\lambda)= \sum _{1 \le i_1 < i_2 <\cdots<i_k\leq
n}\lambda_{i_1}\lambda_{i_2}\cdots\lambda_{i_k}.
\end{equation*}

\begin{definition}
For $1\leq k\leq n$, let $\Gamma_k$ be a cone in $\mathbb{R}^n$ determined by
\begin{equation}\label{gar}
  \Gamma_k  = \{ \lambda  \in \mathbb{R}^n :\sigma _i (\lambda ) >
0,~\forall~ 1 \le i \le k\}.
\end{equation}
\end{definition}

Denote $\sigma_{k-1}(\lambda|i)=\frac{\partial
\sigma_k}{\partial \lambda_i}$ and
$\sigma_{k-2}(\lambda|ij)=\frac{\partial^2 \sigma_k}{\partial
\lambda_i\partial \lambda_j}$. In the following, we list some algebraic equalities and inequalities of $\sigma_k$. Then we introduce some important properties which will be used later. 

\begin{equation}\label{a0}
\sum_{i=1}^{n}\sigma_{k-1}(\lambda|i)=(n-k+1)\sigma_{k-1}(\lambda),
\end{equation}
\begin{equation}\label{a1}
  \sigma_{k-1}(\lambda)\geq c_0\sigma_k^{1-\frac{1}{k-1}}(\lambda)\sigma_1^{\frac{1}{k-1}(\lambda)},
\end{equation}
and
\begin{equation*}
  \sigma_1(\lambda)\geq c_0\sigma_k^{\frac{1}{k}}(\lambda),
\end{equation*}
for any $\lambda\in \Gamma_k$ and some positive constant $c_0$ depending on $n$ and $k$.

\begin{proposition}\label{sigma}
Let $\lambda=(\lambda_1,\dots,\lambda_n)\in\mathbb{R}^n$ and $1\leq
k\leq n$, then we have

(1) $\Gamma_1\supset \Gamma_2\supset \cdot\cdot\cdot\supset
\Gamma_n$;

(2) $\sigma_{k-1}(\lambda|i)>0$ for $\lambda \in \Gamma_k$ and
$1\leq i\leq n$;

(3) $\sigma_k(\lambda)=\sigma_k(\lambda|i)
+\lambda_i\sigma_{k-1}(\lambda|i)$ for $1\leq i\leq n$;

(4)
$\sum_{i=1}^{n}\frac{\partial\left[\frac{\sigma_{k}}{\sigma_{l}}\right]^{\frac{1}{k-l}}}
{\partial \lambda_i}\geq \left[\frac{C^k_n}{C^l_n}\right]^{\frac{1}{k-l}}$ for
$\lambda \in \Gamma_{k}$ and $0\leq l<k$;

(5) $\Big[\frac{\sigma_k}{\sigma_l}\Big]^{\frac{1}{k-l}}$ are
concave in $\Gamma_k$ for $0\leq l<k$;

(6) If $\lambda_1\geq \lambda_2\geq \cdot\cdot\cdot\geq \lambda_n$,
then $\sigma_{k-1}(\lambda|1)\leq \sigma_{k-1}(\lambda|2)\leq
\cdot\cdot\cdot\leq \sigma_{k-1}(\lambda|n)$ for $\lambda \in
\Gamma_k$.

\end{proposition}

\begin{proof}
All the properties are well known. For example, see Chapter XV in
\cite{Li96} or \cite{Hui99} for proofs of (1), (2), (3) and (6); see Lemma 2.2.19 in \cite{Ger06} for the proof of (4); see
\cite{CNS85} and \cite{Li96} for the proof of (5).
\end{proof}

\begin{proposition}
Suppose $W=\{W_{ij}\}$ is diagonal and $m~(1\leq m\leq n)$ is a positive integer, then
\begin{eqnarray*}
\frac{\partial \sigma_m(W)}{\partial W_{ij}}=
  \begin{cases}
  \sigma_{m-1}(W|i), ~&i=j,\\
  0, ~~~&otherwise.
  \end{cases}
\end{eqnarray*}
\begin{eqnarray*}
\frac{\partial^2\sigma_m(W)}{\partial W_{ij}\partial W_{pq}}=
\begin{cases}
\sigma_{m-2}(W|ip), ~~~&i=j,p=q,i\neq p,\\
-\sigma_{m-2}(W|ip), ~~~&i=q,j=p,i\neq j,\\
0,~~~&\mbox{otherwise}.
\end{cases}
\end{eqnarray*}
\end{proposition}

The generalized Newton-MacLaurin inequality is as follows.
\begin{proposition}\label{NM}
For $\lambda \in \Gamma_k$ and $n\geq k > l \geq 0$, $n\geq r > s \geq 0$, $k\geq r$, $l \geq s$, we have
\begin{align}
\Bigg[\frac{{\sigma _k (\lambda )}/{C_n^k }}{{\sigma _l (\lambda
)}/{C_n^l }}\Bigg]^{\frac{1}{k-l}} \le \Bigg[\frac{{\sigma _r
(\lambda )}/{C_n^r }}{{\sigma _s (\lambda )}/{C_n^s
}}\Bigg]^{\frac{1}{r-s}}. \notag
\end{align}
\end{proposition}

\begin{proof}
See \cite{S05}.
\end{proof}

Next, we introduce a new cone $\widetilde{\Gamma}_k$ defined by \eqref{wid} and list some important properties, more details can refer to \cite{CDH23}.
\begin{proposition}
The following properties hold.

(1) $\widetilde{\Gamma}_k$ are convex cones and
\begin{equation}\label{px}
  \Gamma_1=\widetilde{\Gamma}_1\supset\widetilde{\Gamma}_2\supset\cdots\supset\widetilde{\Gamma}_n\supset\Gamma_2.
\end{equation}

(2) If $\lambda=(\lambda_1,\cdots,\lambda_n)\in \widetilde{\Gamma}_k$, then
$$\frac{\partial\left[\frac{\sigma_k(\eta)}{\sigma_l(\eta)}\right]}{\partial\lambda_i}\geq\frac{n(k-l)}{k(n-l)}\sum_{p\neq i}\frac{\sigma_{k-1}(\eta|p)\sigma_l(\eta|p)}{\sigma_l(\eta)^2},$$
for any $i=1,2,\cdots,n$, where $0\leq l<k\leq n$.

(3) If $\lambda=(\lambda_1,\cdots,\lambda_n)\in \widetilde{\Gamma}_k$, then $\left[\frac{\sigma_k(\eta)}{\sigma_l(\eta)}\right]^{\frac{1}{k-l}} (0\leq l<k\leq n)$ are concave with respect to $\lambda$. Hence for any $(\xi_1,\cdots,\xi_n)$, we have
\begin{equation}\label{pro3}
  \sum_{i,j}\frac{\partial^2\left[\frac{\sigma_k(\eta)}{\sigma_l(\eta)}\right]}{\partial \lambda_i\partial \lambda_j}\xi_i\xi_j\leq\left(1-\frac{1}{k-l}\right)\frac{\left[\sum_i\frac{\partial\left(\frac{\sigma_k(\eta)}{\sigma_l(\eta)}\right)}{\partial \lambda_i}\xi_i\right]^2}{\frac{\sigma_k(\eta)}{\sigma_l(\eta)}}.
\end{equation}
\end{proposition}

\begin{proposition}
Suppose that $\lambda=(\lambda_1,\cdots,\lambda_n)\in\widetilde{\Gamma}_k$ are ordered with $\lambda_1\geq\lambda_2\geq\cdots\geq\lambda_n$, then

(1) $\eta_1\leq\eta_2\leq\cdots\leq\eta_n$ and $\eta_{n-k+1}>0$.

(2) $\sigma_{k-1}(\eta|n-k+1)\geq c(n,k)\sigma_{k-1}(\eta),~\mbox{for}~0\leq l<k\leq n$.

(3) $\frac{\partial\left[\frac{\sigma_k(\eta)}{\sigma_l(\eta)}\right]}{\partial\eta_1}\geq
\frac{\partial\left[\frac{\sigma_k(\eta)}{\sigma_l(\eta)}\right]}{\partial\eta_2}\geq\cdots\geq
\frac{\partial\left[\frac{\sigma_k(\eta)}{\sigma_l(\eta)}\right]}{\partial\eta_n},~\mbox{for}~0\leq l<k\leq n$.

(4) $\frac{\partial\left[\frac{\sigma_k(\eta)}{\sigma_l(\eta)}\right]}{\partial\lambda_1}\leq
\frac{\partial\left[\frac{\sigma_k(\eta)}{\sigma_l(\eta)}\right]}{\partial\lambda_2}\leq\cdots\leq
\frac{\partial\left[\frac{\sigma_k(\eta)}{\sigma_l(\eta)}\right]}{\partial\lambda_n},~\mbox{for}~0\leq l<k\leq n$.

(5) For $0\leq l<k<n$, we have
\begin{equation}\label{ile}
  \frac{\partial\left[\frac{\sigma_k(\eta)}{\sigma_l(\eta)}\right]}{\partial\lambda_i}\geq c(n,k,l)\sum_i\frac{\partial\left[\frac{\sigma_k(\eta)}{\sigma_l(\eta)}\right]}{\partial\lambda_i}\geq c(n,k,l)\left[\frac{\sigma_k(\eta)}{\sigma_l(\eta)}\right]^{1-\frac{1}{k-l}},~\forall~1\leq i\leq n.
\end{equation}
\end{proposition}

\subsection{Basic properties of graphic hypersurface}
 Let $\phi_i=\frac{\partial\phi}{\partial x_i}$, $\phi_{ij}=\frac{\partial^2\phi}{\partial x_i\partial x_j}$, $D\phi=(\phi_1,\cdots,\phi_n)$ and $D^2\phi=\{\phi_{ij}\}$ denote the first and second order derivatives, gradient and Hessian of a function $\phi\in C^2(\Omega)$, respectively.
A graphic hypersurface $M_u$ in $\mathbb{R}^{n+1}$ is a codimension one submanifold which can be written as a graph
$$M_u=\{X=(x,u(x))|x\in\mathbb{R}^n\}.$$
Let $\{\epsilon_1,\cdots,\epsilon_{n+1}\}$ be the standard basis of $\mathbb{R}^{n+1}$, then the height function of $M_u$ is $u(x)=\langle X, \epsilon_{n+1}\rangle$. It is easy to see that the induced metric and second fundamental form of $M_u$ are given by
$$g_{ij}=\delta_{ij}+u_iu_j, \quad 1\leq i,j\leq n,$$
and
$$h_{ij}=\frac{u_{ij}}{\sqrt{1+|Du|^2}}.$$
The unit upward normal vector field to $M_u$ is
$$\nu=\frac{(-Du,1)}{\sqrt{1+|Du|^2}}.$$
By straightforward calculations, we have the principle curvature of $M_u$ are eigenvalues of the matrix
$$\frac{1}{\omega}(I-\frac{Du\otimes Du}{\omega^2})D^2u,$$
or the symmetric matrix $A[u]=\{a_{ij}\}$
$$a_{ij}=\frac{1}{\omega}\gamma^{ik}u_{kl}\gamma^{lj},$$
where $\gamma^{ik}=\delta_{ik}-\frac{u_iu_k}{\omega(1+\omega)}$ and $\omega=\sqrt{1+|Du|^2}$. Note that $\{\gamma^{ij}\}$ is invertible with inverse $\gamma_{ij}=\delta_{ij}+\frac{u_iu_j}{1+\omega}$, which is the square root of $\{g_{ij}\}$.

Let $\{e_1,\cdots,e_n\}$ be a local orthonormal frame field defined on $TM_u$ and $\nabla$ be the induced Levi-Civita connection on $M_u$. For a function $v$ on $M_u$, we denote $\nabla_iv=\nabla_{e_i}v$, $\nabla_{ij}v=\nabla^2v(e_i,e_j)$. Hence
$$|\nabla u|=\sqrt{g^{ij}u_iu_j}=\frac{|Du|}{\sqrt{1+|Du|^2}}.$$

We list some well known fundamental equations for a hypersurface $M_u$ in $\mathbb{R}^{n+1}$, which will be needed in our proof.
$$\nabla_{ij}X=h_{ij}\nu\quad(\mbox{Gauss formula}),$$
$$\nabla_i\nu=-h_{ij}e_j\quad(\mbox{Weingarten formula}),$$
$$\nabla_kh_{ij}=\nabla_jh_{ik}\quad(\mbox{Codazzi equation}),$$
$$R_{ijst}=h_{is}h_{jt}-h_{it}h_{js}\quad(\mbox{Gauss equation}),$$
where $h_{ij}=\langle\nabla_{e_i}e_j, \nu\rangle$, $R_{ijst}$ is the $(4,0)$-Riemannian curvature tensor of $M_u$ and the derivative here is covariant derivative with respect to the metric on $M_u$. Hence the Ricci identity becomes
\begin{equation}\label{Ric}
  \nabla_i\nabla_jh_{st}=\nabla_s\nabla_th_{ij}+(h_{mt}h_{sj}-h_{mj}h_{st})h_{mi}+(h_{mt}h_{ij}-h_{mj}h_{it})h_{ms}.
\end{equation}

According to Lemma 3.2 of \cite{GLL}, we state the following proposition by observing the concavity of $(\sigma_k/\sigma_1)^{\frac{1}{k-1}}$, more details see \cite{JS22, JW22}.
\begin{proposition}\label{27}
Let $\alpha=\frac{1}{k-1}$. If $u$ is an admissible solution of Dirichlet problem \eqref{Eq2}, then we get
$$\sum_s\sigma_k^{ij,pq}\nabla_sh_{ij}\nabla_sh_{pq}\leq(1-\alpha)\frac{|\nabla \psi|^2}{\psi}+2\alpha\frac{\langle\nabla \psi, \nabla H\rangle}{H}-(1+\alpha)\frac{\psi|\nabla H|^2}{H^2},$$
where $\sigma_k^{ij,pq}=\frac{\partial^2\sigma_k}{\partial h_{ij}\partial h_{pq}}$.
\end{proposition}

For $p\in \mathbb{R}^n$, define
$$\Gamma_k(p):=\{r\in S^{n\times n}:\lambda(r,p):=\lambda((I-\frac{p\otimes p}{1+|p|^2})r)\in \Gamma_k\}.$$
To establish second order estimates on the boundary, we need the proposition which can be seen in \cite{IV1991,ILT1996}.
\begin{proposition}\label{prop}
For any $p\in \mathbb{R}^n$, we have $\Gamma_{k+1}(0)\subset\Gamma_k(p)$. Assume that $r\in \Gamma_{k+1}$ is an $n\times n$ matrix, then we have for $0\leq l<k<n$,
\begin{equation}\label{pro}
  \frac{\sigma_k}{\sigma_l}(\lambda(r,p))\geq\frac{1}{1+|p|^2}\frac{\sigma_k}{\sigma_l}(\lambda(r)).
\end{equation}
\end{proposition}

\section{$C^1$ estimates}
In this section, we consider $C^1$ estimates for the admissible solution of Dirichlet problem \eqref{ali}. Since $M_u$ is $(\eta,k)$-convex, then $\sigma_1(\lambda(\eta[M_u]))>0$, hence $H[M_u]>0$. Combining with $\psi_z\geq 0$, the subsolution condition \eqref{sub} and the maximum principle, it is easy to derive that
$$\sup_{\overline{\Omega}}|u|+\sup_{\partial\Omega}|Du|\leq C.$$
Then we establish the global gradient estimates.
\begin{theorem}\label{c1}
Let $k\geq 2,~0\leq l<k<n$, $u\in C^3(\Omega)\cap C^1(\overline{\Omega})$ be an admissible solution of Dirichlet problem \eqref{ali}. Suppose that $\psi_z\geq 0$ and $\psi^{\frac{1}{k-l}}\in C^1(\overline{\Omega}\times\mathbb{R}\times\mathbb{S}^n)$. Then there exists a positive constant $C$ depending only on $n, k, l, \|u\|_{C^0}$ and $\|\psi^{\frac{1}{k-l}}\|_{C^1}$ such that
$$\sup_{\overline{\Omega}}|Du|\leq C(1+\sup_{\partial\Omega}|Du|).$$
\end{theorem}

\begin{proof}
Consider the following test function
$$P=\log\omega+B\langle X, \epsilon_{n+1}\rangle=\log \omega+Bu,$$
where $\omega=\frac{1}{\langle\nu, \epsilon_{n+1}\rangle}=\sqrt{1+|Du|^2}$ and $B$ is a positive constant to be determined later. Suppose $P$ attains its maximum at an interior point $x_0\in\Omega$. We rotate the standard basis $\epsilon_1,\cdots, \epsilon_{n}$ to satisfy $u_1=|Du|, u_j=0$ for $j\geq 2$ at $x_0$. Define $e_i=\gamma^{is}\widetilde{\partial}_s, i=1,\cdots,n$, where $\widetilde{\partial}_s:=\epsilon_s+u_s\epsilon_{n+1}$ for $ s= 1,\cdots,n$. Then it is clear that $\{e_1,e_2,,\cdots,e_n\}$ is an orthonormal frame on $M_u$ near $X_0=(x_0,u(x_0))$ satisfying
$$|\nabla u|=\nabla_1u, \quad \nabla_i u=0, \quad \mbox{for}~i\geq 2$$
at $X_0$. Therefore by Weingarten formula,
$$\nabla_i\omega=\frac{h_{im}\nabla_mu}{\langle\nu, \epsilon_{n+1}\rangle^2}=\omega^2h_{im}\nabla_mu.$$
Thus at $X_0$,
$$\nabla_iP=\omega h_{i1}\nabla_1u+B\nabla_iu=0.$$
Taking $i=1$ and $i\geq 2$ respectively, then at $X_0$,
$$\omega h_{11}=-B, \quad h_{i1}=0.$$
We rotate $\{e_2,\cdots,e_n\}$ such that matrix $\{h_{ij}\}$ is diagonal at $X_0$. Thus at $X_0$,
\begin{equation}\label{E1}
\nabla_{ii}P=\omega^2(h_{im}\nabla_mu)^2+\omega\nabla_ih_{i1}\nabla_1u+\omega h_{im}\nabla_{mi}u+B\nabla_{ii}u\leq 0.
\end{equation}
Let $F^{ij}=\frac{\partial\left(\frac{\sigma_k}{\sigma_l}(\lambda(\eta))\right)}{\partial h_{ij}}$ and $\{y_k\}^n_{k=1}$ denote the standard local coordinate system on $\mathbb{S}^n$ near $(0,\cdots,0,1)$. Then differentiating equation \eqref{Eq1} and using Weingarten formula and Codazzi equation, we have at $X_0$,
\begin{eqnarray}\label{E2}
  \nonumber F^{ii}\nabla_ih_{i1} &=&F^{ii}\nabla_1h_{ii}=e_1(\psi) \\
  \nonumber &=&\psi_{x_j}\nabla_1x_j+\psi_z\nabla_1u+\partial_{y_k}\psi\nabla_1\nu_k\\
   &=&\frac{\psi_{x_1}}{\omega}+\psi_z\nabla_1u+\frac{B\partial_{y_1}\psi}{\omega^2}.
\end{eqnarray}
By \eqref{ile}, we know that for $0\leq l<k<n$,
$$F^{11}\geq C_0\psi^{1-\frac{1}{k-l}},$$
where $C_0$ is a constant depending only on $n, k$ and $l$.

According to \eqref{E1}, \eqref{E2} and $\psi_z\geq 0, ~\psi^{\frac{1}{k-l}}\in C^1(\overline{\Omega}\times\mathbb{R}\times\mathbb{S}^n)$, we obtain at $X_0$,
\begin{eqnarray*}
  0 &\geq& \omega^2F^{11}h_{11}^2(\nabla_1u)^2+\omega\nabla_1u(e_1\psi)+\omega^2F^{ii}h_{ii}^2+B(k-l)\psi\omega\\
  &\geq& B^2F^{11}(\nabla_1u)^2+\psi_{x_1}\nabla_1u+\frac{B\partial_{y_1}\psi\nabla_1u}{\omega}\\
  &\geq&(B^2C_0|\nabla u|^2-C|\nabla u|(1+\frac{B}{\omega}))\psi^{1-\frac{1}{k-l}},
\end{eqnarray*}
which implies
$$0\geq B^2C_0|\nabla u|-C(1+\frac{B}{\omega}).$$
By choosing $B$ sufficiently large, the proof of Theorem \ref{c1} is completed.
\end{proof}

\section{Global second order derivative estimates}
\begin{theorem}\label{a}
Suppose $\psi^{\frac{1}{k-l}}\in C^{1,1}(\overline{\Omega}\times \mathbb{R}\times\mathbb{S}^n)$. Let $k\geq 2, 0\leq l<k<n$, $u\in C^4(\Omega)\cap C^2(\overline{\Omega})$ be an admissible solution of Dirichlet problem \eqref{ali}. Then there exists a positive constant $C$ depending on $n,k,l,\|u\|_{C^1},\|\psi^{\frac{1}{k-l}}\|_{C^{1,1}}$ such that
$$\sup_{\overline{\Omega}}|D^2u|\leq C(1+\sup_{\partial\Omega}|D^2u|).$$
\end{theorem}

\begin{proof}
Let $v:=\frac{1}{\omega}=\langle\nu, \epsilon_{n+1}\rangle$. There exists a positive constant $a$ depending only on $\|Du\|_{C^0}$ such that $v\geq 2a$. Assume that $\kappa_{max}$ is the largest principle curvature, then we consider the auxiliary function
$$Q=\log \kappa_{max}-\log(v-a).$$
Suppose that $Q$ attains its maximum at a point $X_0=(x_0,u(x_0))\in M_u, x_0\in \Omega$, then we can choose a local orthonormal frame $\{e_1,e_2,\cdots,e_n\}$ near $X_0$ such that
$$h_{ij}=h_{ii}\delta_{ij},\quad h_{11}\geq h_{22}\geq\cdots \geq h_{nn}$$
at $X_0$. Then $\eta_{ij}=Hg_{ij}-h_{ij}$ is also diagonal at $X_0$ and
$$\eta_{11}\leq\eta_{22}\leq\cdots\leq\eta_{nn}.$$
We define a new function $\widetilde{Q}$ by
$$\widetilde{Q}=\log h_{11}-\log(v-a).$$
It obvious that $\widetilde{Q}$ also attains its maximum value at $X_0$. Denote
$$F=\frac{\sigma_k(\lambda(\eta))}{\sigma_l(\lambda(\eta))},\quad F^{ij}=\frac{\partial F}{\partial h_{ij}}, \quad F^{ij,rs}=\frac{\partial^2 F}{\partial h_{ij}\partial h_{rs}}.$$
From now on, all calculations are performed at $X_0$. Then
\begin{equation}\label{Qi}
  0=\nabla_i\widetilde{Q}=\frac{\nabla_ih_{11}}{h_{11}}-\frac{\nabla_iv}{v-a},
\end{equation}
\begin{equation}\label{Qii}
  0\geq F^{ii}\nabla_{ii}\widetilde{Q}=\frac{F^{ii}\nabla_{ii}h_{11}}{h_{11}}-\frac{F^{ii}(\nabla_ih_{11})^2}{h_{11}^2}+\frac{F^{ii}(\nabla_i v)^2}{(v-a)^2}-\frac{F^{ii}\nabla_{ii}v}{v-a}.
\end{equation}
By Weingarten formula,
\begin{equation}\label{vi}
  \nabla_i v=-h_{im}\langle e_m, \epsilon_{n+1}\rangle=-h_{im}\nabla_mu, \quad F^{ii}(\nabla_i v)^2\leq CF^{ii}h_{ii}^2.
\end{equation}
And by Gauss formula, Codazzi equation and $\psi^{\frac{1}{k-l}}\in C^1(\overline{\Omega}\times \mathbb{R}\times\mathbb{S}^n)$,
\begin{eqnarray}\label{nu}
 \nonumber F^{ii}\nabla_{ii}v &=& -F^{ii}\nabla_mh_{ii}\nabla_mu-F^{ii}h_{ii}^2v \\
  \nonumber &=&-\langle\nabla \psi,\nabla u\rangle-vF^{ii}h_{ii}^2\\
  &\leq&-vF^{ii}h_{ii}^2+C\psi^{1-\frac{1}{k-l}},
\end{eqnarray}
where $C$ is a constant depending on $\|\psi^{\frac{1}{k-l}}\|_{C^1}$ and $\|u\|_{C^1}$.

Differentiating equation \eqref{Eq1} twice and using \eqref{Ric},
\begin{eqnarray}
  \nonumber F^{ii}\nabla_{ii}h_{11} &=& F^{ii}\nabla_{11}h_{ii}+F^{ii}(h_{i1}^2-h_{ii}h_{11})h_{ii}+F^{ii}(h_{ii}h_{11}-h_{i1}^2)h_{11}\\
 \nonumber &=& F^{ii}\nabla_{11}h_{ii}-F^{ii}h_{ii}^2h_{11}+F^{ii}h_{ii}h_{11}^2\\
  &=&\nabla_{11}\psi-F^{ij,rs}\nabla_1h_{ij}\nabla_1h_{rs}-F^{ii}h_{ii}^2h_{11}+(k-l)\psi h_{11}^2.
\end{eqnarray}
Let $\widetilde{\psi}=\psi^{\alpha}, \alpha=\frac{1}{k-l}$. Assume that $h_{11}\geq 1$. Then by direct calculations,
\begin{equation}\label{e1}
  \nabla_{11}\psi\geq-C\psi^{1-\alpha}h_{11}^2-\frac{1}{\alpha}\sum_p\nabla_ph_{11}(d_{\nu}\widetilde{\psi})(e_p)\psi^{1-\alpha}+(1-\alpha)
\frac{(\nabla_1\psi)^2}{\psi}.
\end{equation}
Combining with \eqref{Qi} and \eqref{vi},
\begin{equation}\label{e2}
  \frac{\nabla_ph_{11}}{h_{11}}(d_{\nu}\widetilde{\psi})(e_p)\psi^{1-\alpha}\leq Ch_{11}\psi^{1-\alpha}.
\end{equation}
By \eqref{pro3},
\begin{equation}\label{e3}
  F^{ij,rs}\nabla_1h_{ij}\nabla_1h_{rs}\leq(1-\frac{1}{k-l})\frac{(\nabla_1\psi)^2}{\psi}.
\end{equation}
According to \eqref{e1}-\eqref{e3},
\begin{equation}\label{e4}
  \frac{F^{ii}\nabla_{ii}h_{11}}{h_{11}}\geq-Ch_{11}\psi^{1-\alpha}-F^{ii}h_{ii}^2+(k-l)\psi h_{11}.
\end{equation}
Combining \eqref{ile}, \eqref{Qi}, \eqref{Qii}, \eqref{nu} and \eqref{e4},
\begin{eqnarray*}
  0 &\geq& \frac{a}{v-a}F^{ii}h_{ii}^2-\frac{C}{v-a}\psi^{1-\frac{1}{k-l}}-Ch_{11}\psi^{1-\frac{1}{k-l}}+(k-l)\psi h_{11} \\
  &\geq& \frac{a}{v-a}F^{11}h_{11}^2-Ch_{11}\psi^{1-\frac{1}{k-l}}-C\psi^{1-\frac{1}{k-l}}\\
  &\geq&\left(\frac{a}{v-a}c(n,k,l)h_{11}^2-Ch_{11}-C\right)\psi^{1-\frac{1}{k-l}},
\end{eqnarray*}
which implies $h_{11}\leq C$. Then the proof is completed.
\end{proof}

\section{Boundary second order derivative estimates}
Since equation \eqref{Eq1} can be written as
\begin{equation}\label{geq}
  G(D^2u, Du)=\frac{\sigma_k}{\sigma_l}(\lambda\{b_{ij}\})=f(\lambda(A))=\psi(x,u,Du),
\end{equation}
where $G=G(r,p)$ is viewed as a function of $(r,p)$ for $r\in\mathbb{S}^{n\times n}, p\in \mathbb{R}^n$, $\{b_{ij}\}=T(A)=(trace A)I-A$ with $A=\{a_{ij}\}$. Define
$$G^{ij}=\frac{\partial G}{\partial r_{ij}}(D^2u, Du), \quad G^i=\frac{\partial G}{\partial p_i}(D^2u, Du), \quad \psi_{u_i}=\frac{\partial \psi}{\partial u_i}(x,u,Du),$$
and
$$L=G^{ij}\partial_{ij}-\psi_{u_i}\partial_i.$$
We need the following lemma in \cite{GS04}.
\begin{lemma}\label{lem}
We have
$$G^s=-\frac{u_s}{\omega^2}\sum_if_i\kappa_i-\frac{2}{\omega(\omega+1)}\sum_{t,j}F^{ij}a_{it}(\omega u_t\gamma^{sj}+u_j\gamma^{ts}),$$
where $a_{ij}=\frac{1}{\omega}\gamma^{ik}u_{kl}\gamma^{lj}$, $\kappa=\lambda(\{a_{ij}\})$, $f_i=\frac{\partial f}{\partial \kappa_i}$ and $F^{ij}=\frac{\partial f(\lambda(A))}{\partial a_{ij}}$.
\end{lemma}

Define
\begin{equation}\label{w}
  W:=\nabla'_{\alpha}u-\frac{1}{2}\sum_{\beta\leq n-1}u_{\beta}^2,~~\mbox{on}~~\overline{\omega}_{\delta}
\end{equation}
for some small $\delta$, where
$$\nabla'_{\alpha}u:=u_{\alpha}+\rho_{\alpha}u_n,~~\mbox{for}~~1\leq\alpha\leq n-1.$$
Then we prove an important lemma, which will be used to derive tangential-normal estimates.
\begin{lemma}\label{im}
If $\delta$ is sufficiently small and $\psi^{\frac{1}{k-l}}\in C^1(\overline{\Omega}\times\mathbb{R}\times\mathbb{S}^n)$, then
\begin{equation}\label{key}
  LW\leq C\left(\psi^{1-\frac{1}{k-l}}+\psi|DW|+\sum_iG^{ii}+G^{ij}W_iW_j\right),
\end{equation}
where $C$ is a positive constant depending on $n,k,l,\|u\|_{C^1},\|\psi^{\frac{1}{k-l}}\|_{C^1}$ and $\partial\Omega$.
\end{lemma}
\begin{proof}
By \eqref{w} and differentiating equation \eqref{geq},
\begin{eqnarray}\label{wij}
 \nonumber G^{ij}W_{ij}+G^sW_s &=& \nabla'_{\alpha}\psi-\sum_{\beta\leq n-1}u_{\beta}\psi_{\beta}-\sum_{\beta\leq n-1}G^{ij}u_{\beta i}u_{\beta j} \\
  &&+2G^{ij}u_{ni}\rho_{\alpha j}+u_nG^{ij}\rho_{\alpha ij}+u_nG^s\rho_{\alpha s}.
\end{eqnarray}
Combining \eqref{wij} and $\psi^{\frac{1}{k-l}}\in C^1(\overline{\Omega}\times\mathbb{R}\times\mathbb{S}^n)$, it is easy to derive
\begin{eqnarray}\label{lw}
 \nonumber LW+G^sW_s&\leq&C\psi^{1-\frac{1}{k-l}}+2G^{ij}u_{ni}\rho_{\alpha j}\\
  &&-\sum_{\beta\leq n-1}G^{ij}u_{\beta i}u_{\beta j}+u_nG^{ij}\rho_{\alpha ij}+u_nG^s\rho_{\alpha s}.
\end{eqnarray}
Since
$$G^{ij}=\frac{1}{\omega}\sum_{s,t}\gamma^{is}F^{st}\gamma^{tj}~~\mbox{and}~~u_{ij}=\omega\sum_{s,t}\gamma_{is}a_{st}\gamma_{tj},$$
it follows that
$$\sum_{\beta\leq n-1}G^{ij}u_{\beta i}u_{\beta j}=\omega\sum_{\beta\leq n-1}\sum_{s,t}F^{ij}\gamma_{\beta s}\gamma_{\beta t}a_{si}a_{tj}.$$
By \cite{GS04}, we choose an orthogonal matrix $B=\{b_{ij}\}$ such that $\{a_{ij}\}$ and $\{F^{ij}\}$ are diagonalized at the same time
$$F^{ij}=\sum_sb_{is}f_sb_{js}~~\mbox{and}~~a_{ij}=\sum_sb_{is}\kappa_sb_{js}.$$
Then
\begin{equation*}
  \sum_{\beta\leq n-1}G^{ij}u_{\beta i}u_{\beta j}=\omega\sum_{\beta\leq n-1}\sum_i\left(\sum_s\gamma_{\beta s}b_{si}\right)^2f_i\kappa_i^2.
\end{equation*}
Let the matrix $\eta=\{\eta_{ij}\}=\{\sum_s\gamma_{is}b_{sj}\}$. It is clear that $\eta\cdot\eta^{T}=g$ and $|\det(\eta)|=\sqrt{1+|Du|^2}$. Therefore
\begin{equation}\label{gij}
\sum_{\beta\leq n-1}G^{ij}u_{\beta i}u_{\beta j}=\omega\sum_{\beta\leq n-1}\sum_i\eta_{\beta i}^2f_i\kappa_i^2.
\end{equation}
We also have
\begin{equation}\label{uni}
  G^{ij}u_{ni}\rho_{\alpha j}=\sum_{i,t}f_i\kappa_ib_{si}\gamma^{js}b_{ti}\gamma_{nt}\rho_{\alpha j}\leq C\sum_if_i|\kappa_i|.
\end{equation}
For any indices $j,t$,
$$F^{ij}a_{it}=\sum_ib_{ti}f_i\kappa_ib_{ij}\leq\sum_if_i|\kappa_i|.$$
From Lemma \ref{lem},
\begin{equation}\label{gs}
  |G^s\rho_{\alpha s}|\leq C\sum_if_i|\kappa_i|.
\end{equation}
By \eqref{lw}-\eqref{gs},
\begin{equation}\label{lwg}
  LW+G^sW_s\leq C(\psi^{1-\frac{1}{k-l}}+\sum_iG^{ii}+\sum_if_i|\kappa_i|)-\omega\sum_{\beta\leq n-1}\sum_i\eta_{\beta i}^2f_i\kappa_i^2.
\end{equation}
For the term $G^sW_s$, by Lemma \ref{lem} and the definition of the matrix $\{b_{ij}\}$,
\begin{eqnarray}\label{gsws}
 \nonumber -G^sW_s&=&\frac{1}{\omega}\sum_s\left(\frac{(k-l)\psi u_s}{\omega}+2\sum_{t,i}f_i\kappa_i(b_{ti}u_t)\gamma^{sj}b_{ji}\right)W_s\\
  &\leq&C\psi|DW|+\frac{2}{\omega}\sum_{t,i}f_i\kappa_i(b_{ti}u_t)\gamma^{sj}b_{ji}W_s.
\end{eqnarray}
Next we divide into two cases as follows:

$\mathbf{Case (a):}$~~$\sum_{\beta\leq n-1}\eta_{\beta i}^2\geq \epsilon$~~for all $i$.

By \eqref{gij},
\begin{equation}\label{bet}
  \sum_{\beta\leq n-1}G^{ij}u_{\beta i}u_{\beta j}\geq\epsilon\sum_if_i\kappa_i^2.
\end{equation}
Using Cauchy-Schwarz inequality, we get
\begin{equation}\label{cs}
  \frac{2}{\omega}\kappa_i(b_{ti}u_t)\gamma^{sj}b_{ji}W_s\leq\frac{\epsilon}{2}\kappa_i^2+\frac{C}{\epsilon}(\gamma^{sj}b_{ji}W_s)^2.
\end{equation}
Combining with \eqref{gsws} and \eqref{cs},
\begin{eqnarray}\label{ewq}
 \nonumber -G^sW_s&\leq& C\psi|DW|+\frac{\epsilon}{2}\sum_if_i\kappa_i^2+\frac{C}{\epsilon}f_i(\gamma^{sj}b_{ji}W_s)(\gamma^{tm}b_{mi}W_t)\\
  &\leq& C\psi|DW|+\frac{\epsilon}{2}\sum_if_i\kappa_i^2+\frac{C}{\epsilon}G^{ij}W_iW_j.
\end{eqnarray}
For any $\overline{\varepsilon}>0$,
\begin{equation}\label{sui}
  \sum_if_i|\kappa_i|\leq\overline{\varepsilon}\sum_if_i\kappa_i^2+\frac{C}{\overline{\varepsilon}}\sum_iG^{ii}.
\end{equation}
By \eqref{lwg}-\eqref{sui} and choosing $\overline{\varepsilon}\leq\frac{\epsilon}{2}$, it is easy to derive \eqref{key}.

$\mathbf{Case (b):}$~~$\sum_{\beta\leq n-1}\eta_{\beta r}^2< \epsilon$~~for some index $1\leq r\leq n$, where $\epsilon$ is some positive constant to be determined later.

As in \cite{IV1990},
$$1\leq\det(\eta)\leq\eta_{nr}\det(\eta')+C_1\epsilon\leq\sqrt{1+\mu_1^2}|\det(\eta')|+C_1\epsilon,$$
where $\eta':=\{\eta_{\alpha\beta}\}_{\alpha\neq n,\beta\neq r}, \mu_1:=\|Du\|_{C^0}$ and $C_1$ is a positive constant depending on $n$ and $\mu_1$. Set $\epsilon$ small enough such that $C_1\epsilon<\frac{1}{2}$, then
$$|\det(\eta')|\geq\frac{1}{2\sqrt{1+\mu_1^2}}.$$
On the other hand, for any fixed $\alpha\neq r$, we derive
$$|\det(\eta')|\leq C\sum_{\beta\leq n-1}|\eta_{\alpha\beta}|\leq\sqrt{(n-1)\sum_{\beta\leq n-1}\eta_{\alpha\beta}^2}.$$
Hence for any $i\neq r$,
$$\sum_{\beta\leq n-1}\eta_{\beta i}^2\geq c_1,$$
where $c_1$ is a positive constant  depending on $n$ and $\mu_1$. In view of \eqref{gij},
\begin{equation}\label{beta}
  \sum_{\beta\leq n-1}G^{ij}u_{\beta i}u_{\beta j}\geq c_1\sum_{i\neq r}f_i\kappa_i^2.
\end{equation}
Next we divide into two cases of $\kappa_r\leq 0$ and $\kappa_r> 0$.

$\mathbf{Case (b1):}$ $\kappa_r\leq 0$.

As Lemma 2.20 of \cite{Guan14}, then
$$\sum_{i\neq r}f_i\kappa_i^2\geq\frac{1}{n+1}\sum_if_i\kappa_i^2.$$
Similar to Case (a), we can prove \eqref{key}.

$\mathbf{Case (b2):}$ $\kappa_r> 0$.

Without loss of generality, assume that $\kappa_1\geq\cdots\geq\kappa_n$. Let  $\lambda=(\lambda_1,\cdots,\lambda_n)$ be the eigenvalues of $\eta[M_u]$, then $\lambda_i=\sum_j\kappa_j-\kappa_i$. It follows that $\lambda_1\leq\cdots\leq\lambda_n$. Then we consider two subcases.

$\mathbf{Case (b2-1):}$ $\kappa_n\geq -\epsilon_0\kappa_r$, where $\epsilon_0$ is a positive constant to be chosen later.

In order to derive \eqref{key}, the key is to prove the important inequality
\begin{equation}\label{claim}
  f_r\kappa_r\leq C\psi.
\end{equation}

\emph{Proof of \eqref{claim}:}
 If $\kappa_n\geq -\epsilon_0\kappa_r$, then for any $i\neq r$, by choosing $\epsilon_0\leq\frac{1}{2(n-2)}$ we derive
\begin{equation}\label{b21}
  \lambda_i=\sum_j\kappa_j-\kappa_i\geq[1-(n-2)\epsilon_0]\kappa_r\geq\frac{1}{2}\kappa_r>0.
\end{equation}

When $r\neq 1$, it is obvious that
$$\lambda_n\geq\cdots\geq\lambda_1\geq[1-(n-2)\epsilon_0]\kappa_r\geq\frac{1}{2}\kappa_r>0.$$
Thus
\begin{eqnarray}\label{fr}
 \nonumber f_r\kappa_r &=& \sum_{p\neq r}\frac{\sigma_{k-1}(\eta|p)\sigma_l(\eta)-\sigma_k(\eta)\sigma_{l-1}(\eta|p)}{\sigma_l(\eta)^2}\kappa_r\\
 \nonumber &\leq&(n-k+1)\frac{\sigma_{k-1}(\eta)}{\sigma_l(\eta)}\kappa_r\\
  \nonumber&\leq&\frac{(n-k+1)C_n^{k-1}}{\sigma_l(\eta)}\lambda_n\cdots\lambda_{n-k+2}\kappa_r\\
  \nonumber&\leq&\frac{2(n-k+1)C_n^{k-1}}{\sigma_l(\eta)}\lambda_n\cdots\lambda_{n-k+2}\lambda_{n-k+1}\\
  &\leq&C\frac{\sigma_k(\eta)}{\sigma_l(\eta)}=C\psi.
\end{eqnarray}

When $r=1$, i.e., $\kappa_n\geq -\epsilon_0\kappa_1$. By \eqref{b21},
$$\lambda_i\geq[1-(n-2)\epsilon_0]\kappa_1>0,\quad \forall~i\geq 2.$$
Similar to \eqref{fr}, it is easy to derive $f_r\kappa_r\leq C\psi$ if $\lambda_1\geq 0$. Hence we only consider the case $\lambda_1<0$, then
$$0>\lambda_1=\kappa_2+\cdots+\kappa_n\geq(n-1)\kappa_n\geq-(n-1)\epsilon_0\kappa_1.$$
By $\lambda_1<0$, we get $\sigma_{k-1}(\eta)\leq\sigma_{k-1}(\eta|1)$. Thus
\begin{eqnarray}\label{kap}
 \nonumber f_1\kappa_1 &\leq&(n-k+1)\frac{\sigma_{k-1}(\eta)}{\sigma_l(\eta)}\kappa_1 \\
 \nonumber &\leq&(n-k+1)\frac{\sigma_{k-1}(\eta|1)}{\sigma_l(\eta)}\kappa_1\\
  &\leq&\frac{(n-k+1)C_{n-1}^{k-1}}{[1-(n-2)\epsilon_0]\sigma_l(\eta)}\lambda_n\cdots\lambda_{n-k+2}\lambda_{n-k+1}.
\end{eqnarray}
Note that
\begin{eqnarray}\label{sig}
 \nonumber \sigma_k(\eta) &=&\sigma_k(\eta|1)+\lambda_1\sigma_{k-1}(\eta|1)\\
 \nonumber &\geq&\lambda_n\cdots\lambda_{n-k+1}+\lambda_1\cdot C_{n-1}^{k-1}\lambda_n\cdots\lambda_{n-k+2}\\
  &=&(\lambda_{n-k+1}+C_{n-1}^{k-1}\lambda_1)\lambda_n\cdots\lambda_{n-k+2}.
\end{eqnarray}
Then
$$\lambda_{n-k+1}\geq[1-(n-2)\epsilon_0]\kappa_1\geq-\frac{1-(n-2)\epsilon_0}{(n-1)\epsilon_0}\lambda_1,$$
if $n-k+1\geq 2$. Therefore by choosing $\epsilon_0\leq \frac{1}{n-2+2(n-1)C_{n-1}^{k-1}}$ sufficiently small, we get
\begin{equation}\label{12}
  \frac{1}{2}\lambda_{n-k+1}+C_{n-1}^{k-1}\lambda_1\geq-\lambda_1\left(\frac{1-(n-2)\epsilon_0}{2(n-1)\epsilon_0}-C_{n-1}^{k-1}\right)\geq 0.
\end{equation}
Combining with \eqref{kap}-\eqref{12}, \eqref{claim} still holds.

Similar to \eqref{ewq},
\begin{eqnarray}\label{csin}
  \nonumber-G^sW_s &\leq& C\psi|DW|+\frac{2}{\omega}\sum_{t,i}f_i\kappa_i(b_{ti}u_t)\gamma^{sj}b_{ji}W_s\\
  \nonumber&\leq&C\psi|DW|+\frac{2}{\omega}\sum_tf_r\kappa_r(b_{tr}u_t)\gamma^{sj}b_{jr}W_s+\frac{C}{\omega}\sum_{i\neq r}f_i|\kappa_i||\gamma^{sj}b_{ji}W_s| \\
  \nonumber&\leq& C\psi|DW|+\epsilon\sum_{i\neq r}f_i\kappa_i^2+\frac{C}{\epsilon}\sum_if_i(\gamma^{sp}b_{pi}W_s)(\gamma^{tq}b_{qi}W_t)\\
  &\leq& C\psi|DW|+\epsilon\sum_{i\neq r}f_i\kappa_i^2+\frac{C}{\epsilon}G^{ij}W_iW_j,
\end{eqnarray}
where the third inequality comes from \eqref{claim}.

By \eqref{claim},
\begin{equation}\label{fik}
  \sum_if_i|\kappa_i|=f_r\kappa_r+\sum_{i\neq r}f_i|\kappa_i|\leq C\psi+ \widetilde{\varepsilon}\sum_{i\neq r}f_i\kappa_i^2+\frac{C}{\widetilde{\varepsilon}}\sum_iG^{ii},
\end{equation}
for any $\widetilde{\varepsilon}>0$.
Then using \eqref{lwg}, \eqref{beta}, \eqref{csin} and \eqref{fik}, \eqref{key} is proved by choosing $\epsilon+\widetilde{\varepsilon}\leq c_1$.

$\mathbf{Case (b2-2):}$ $\kappa_n<-\epsilon_0\kappa_r$.

Since $\kappa_r>0$, it is obvious that $r\neq n$.
In this subcase, we get
\begin{eqnarray*}
  \gamma^{sj}b_{jr}W_s &=& \gamma^{sj}b_{jr}(u_{\alpha s}+\rho_{\alpha s}u_n+\rho_{\alpha}u_{ns}-\sum_{\beta\leq n-1}u_{\beta}u_{\beta s}) \\
  &=&\omega\kappa_r(\eta_{\alpha r}+\rho_{\alpha}\eta_{nr}-\sum_{\beta\leq n-1}u_{\beta}\eta_{\beta r})+\gamma^{sj}b_{jr}\rho_{\alpha s}u_n.
\end{eqnarray*}
It follows that
\begin{equation}\label{fol}
  |\gamma^{sj}b_{jr}W_s|\leq C\omega\kappa_r(\epsilon+|\rho_{\alpha}|)+C.
\end{equation}
It is obvious that $f_r\kappa_r=(k-l)\psi-\sum_{i\neq r}f_i\kappa_i$, thus by $\kappa_n^2>\epsilon_0^2\kappa_r^2$,
\begin{eqnarray}\label{fra}
 \nonumber \frac{2}{\omega}f_r\kappa_r|(\sum_tb_{tr}u_t)\gamma^{sj}b_{jr}W_s| &=&\frac{2}{\omega}[(k-l)\psi-\sum_{i\neq r}f_i\kappa_i]|(\sum_tb_{tr}u_t)\gamma^{sj}b_{jr}W_s| \\
 \nonumber &\leq& C\psi|DW|+C(\epsilon+|\rho_{\alpha}|)\sum_{i\neq r}f_i|\kappa_i|\kappa_r+C\sum_{i\neq r}f_i|\kappa_i|\\
 \nonumber &\leq& C\psi|DW|+C\epsilon_0^{-1}(\epsilon+|\rho_{\alpha}|)\sum_{i\neq r}f_i\kappa_i^2\\
 \nonumber &&+C\epsilon_0(\epsilon+|\rho_{\alpha}|)\sum_{i\neq r}f_i\kappa_r^2+\epsilon_1\sum_{i\neq r}f_i\kappa_i^2+\frac{C}{\epsilon_1}\sum_{i\neq r}f_i\\
  &\leq&C\psi|DW|+(C\epsilon_0^{-1}(\epsilon+|\rho_{\alpha}|)+\epsilon_1)\sum_{i\neq r}f_i\kappa_i^2+\frac{C}{\epsilon_1}\sum_iG^{ii},
\end{eqnarray}
for any $\epsilon_1>0$. Here the last inequality comes from $\sum_{i\neq r}f_i\kappa_r^2\leq \frac{n-1}{\epsilon_0^2}f_n\kappa_n^2\leq \frac{n-1}{\epsilon_0^2}\sum_{i\neq r}f_i\kappa_i^2$ with $r\neq n$.

Inserting \eqref{fra} into \eqref{gsws},
\begin{eqnarray}\label{eps}
  \nonumber-G^sW_s &\leq&C\psi|DW|+\frac{C}{\omega}\sum_{i\neq r}f_i|\kappa_i||\gamma^{sj}b_{ji}W_s|\\
 \nonumber &&+(C\epsilon_0^{-1}(\epsilon+|\rho_{\alpha}|)+\epsilon_1)\sum_{i\neq r}f_i\kappa_i^2+\frac{C}{\epsilon_1}\sum_iG^{ii} \\
  \nonumber&\leq& C\psi|DW|+(C\epsilon_0^{-1}(\epsilon+|\rho_{\alpha}|)+2\epsilon_1)\sum_{i\neq r}f_i\kappa_i^2\\
  &&+\frac{C}{\epsilon_1}G^{ij}W_iW_j+\frac{C}{\epsilon_1}\sum_iG^{ii}.
\end{eqnarray}
Then using \eqref{gij}, \eqref{lwg}, \eqref{beta}, \eqref{fik} and \eqref{eps}, we derive
\begin{eqnarray*}
  LW&\leq&-G^sW_s+C(\psi^{1-\frac{1}{k-l}}+\sum_iG^{ii}+\sum_if_i|\kappa_i|)-\omega\sum_{\beta\leq n-1}\sum_i\eta_{\beta i}^2f_i\kappa_i^2 \\
  &\leq& C\psi|DW|+\frac{C}{\epsilon_1}G^{ij}W_iW_j+(C\epsilon_0^{-1}(\epsilon+|\rho_{\alpha}|)+C\epsilon_1-c_1)\sum_{i\neq r}f_i\kappa_i^2\\
  &&+C\psi^{1-\frac{1}{k-l}}+(C+\frac{C}{\epsilon_1})\sum_iG^{ii}\\
  &\leq&C(\psi^{1-\frac{1}{k-l}}+\psi|DW|+\sum_iG^{ii}+G^{ij}W_iW_j),
\end{eqnarray*}
by choosing $\epsilon, \delta, \epsilon_1$ sufficiently small such that
$C\epsilon_0^{-1}(\epsilon+|\rho_{\alpha}|)+C\epsilon_1\leq c_1$.
\end{proof}

Then we establish boundary second order derivative estimates.
\begin{theorem}\label{b}
Suppose $\Omega$ is a uniformly $k$-convex bounded domain in $\mathbb{R}^n$ with $\partial \Omega\in C^{2,1}$, $\psi^{\frac{1}{k-l}}\in C^1(\overline{\Omega}\times\mathbb{R}\times\mathbb{S}^n)$. Let $k\geq 2,~0\leq l<k<n$, $u\in C^3(\overline{\Omega})$ be an admissible solution of Dirichlet problem \eqref{ali}, then there exists a positive constant $C$ depending on $n,k,l$, $\|u\|_{C^1}$, $\|\psi^{\frac{1}{k-l}}\|_{C^1}$ and $\partial \Omega$ such that
\begin{equation}\label{be}
  \max_{\partial\Omega}|D^2 u|\leq C.
\end{equation}
\end{theorem}

\begin{proof}
For an arbitrary point $x\in\partial \Omega$, without loss of generality, we may assume that $x$ is the origin and that the positive $x_n$-axis in the interior normal direction to $\partial \Omega$ at the origin. For convenience, in the following we use the notation $C$ to represent some positive constant depending on $n, k, l, \|u\|_{C^1}$, $\|\psi^{\frac{1}{k-l}}\|_{C^1}$ and $\partial \Omega$. The proof will be divided into three steps.

$\mathbf{Step~1:}$ Estimates of $u_{\alpha\beta}(0),~\alpha,\beta=1,\cdots,n-1$.

Near the origin, the boundary $\partial\Omega$ is represented by
\begin{equation}\label{xn}
  x_n=\rho(x')=\frac{1}{2}\sum_{\alpha<n}\kappa_{\alpha}^bx_{\alpha}^2+O(|x'|^3),
\end{equation}
where $\kappa_1^b,\cdots,\kappa_{n-1}^b$ are the principal curvatures of $\partial \Omega$ at the origin and $x'=(x_1,\cdots,x_{n-1})$. Differentiating the boundary condition $u=0$ on $\partial\Omega$ twice, then
$$|u_{\alpha\beta}(0)|\leq C, \quad 1\leq \alpha,\beta\leq n-1.$$

$\mathbf{Step~2:}$ Estimates of $u_{\alpha n}(0),~\alpha=1,\cdots,n-1$.

Let $\omega_{\delta}=\{x\in \Omega:\rho(x')<x_n<\rho(x')+\delta^2,|x'|<\delta\}$. Since $\Omega$ is uniformly $k$-convex, there exist two positive constants $\theta$ and $K$ satisfying
\begin{equation}\label{ga}
  (\kappa_1^b-3\theta,\cdots,\kappa_{n-1}^b-3\theta,2K)\in \Gamma_{k+1}.
\end{equation}
Define
\begin{equation}\label{v}
  v=\rho(x')-x_n-\theta|x'|^2+Kx_n^2.
\end{equation}
Note that the boundary $\partial\omega_{\delta}$ consists of three parts: $\partial\omega_{\delta}=\partial_1\omega_{\delta}\cup\partial_2\omega_{\delta}\cup\partial_3\omega_{\delta}$, where $\partial_1\omega_{\delta}$, $\partial_2\omega_{\delta}$ are defined by $\{x_n=\rho\}\cap\overline{\omega}_{\delta}$, $\{x_n=\rho+\delta^2\}\cap\overline{\omega}_{\delta}$ respectively, and $\partial_3\omega_{\delta}$ is defined by $\{|x'|=\delta\}\cap\overline{\omega}_{\delta}$.
When $\delta$ depending on $\theta$ and $K$ is sufficiently small, we have
\begin{equation}\label{vle}
  \begin{split}
  &v\leq-\frac{\theta}{2}|x'|^2,\quad\mbox{on}~\partial_1\omega_{\delta},\\
  &v\leq-\frac{\delta^2}{2},~~~\quad\quad\mbox{on}~\partial_2\omega_{\delta},\\
  &v\leq-\frac{\theta\delta^2}{2},\quad\quad\mbox{on}~\partial_3\omega_{\delta}.
  \end{split}
\end{equation}
In view of \eqref{xn} and \eqref{ga}, $v$ is $(k+1)$-convex on $\overline{\omega}_{\delta}$. Thus there exists a uniform constant $\eta_0>0$ depending only on $\theta, \partial \Omega$ and $K$ satisfying
$$\lambda(D^2v-2\eta_0I)\in \Gamma_{k+1}~~\mbox{for}~~k<n~~~\mbox{and}~~\lambda(D^2v-2\eta_0I)\in\Gamma_n~~\mbox{for}~~k=n,~~\mbox{on}~~\overline{\omega}_{\delta}.$$
By Proposition \ref{prop},
\begin{equation}\label{lam}
  \lambda\left(\frac{1}{\omega}\{\gamma^{is}(v_{st}-2\eta_0\delta_{st})\gamma^{jt}\}\right)\in \Gamma_k,~~\mbox{on}~~\overline{\omega}_{\delta}.
\end{equation}
Then we consider the following barrier function on $\overline{\omega}_{\delta}$ for sufficiently small $\delta$,
\begin{equation}\label{psi}
  \Psi:=v-td+\frac{N}{2}d^2,
\end{equation}
where $v(x)$ is defined as \eqref{v}, $d(x):=\mbox{dist}(x,\partial \Omega)$ is the distance from $x$ to the boundary $\partial \Omega$, $t, N$ are two positive constants to be determined later.

Let
\begin{equation}\label{ww}
  \widetilde{W}=1-e^{-bW},
\end{equation}
where $W$ is defined as \eqref{w}, $b$ is a sufficiently large positive constant to be determined later.

Using \eqref{key}, we obtain
\begin{eqnarray*}
  L\widetilde{W}&=&be^{-bW}LW-b^2e^{-bW}G^{ij}W_iW_j\\
  &\leq&Cbe^{-bW}(\psi^{1-\frac{1}{k-l}}+\psi|DW|+\sum_iG^{ii}+G^{ij}W_iW_j)-b^2e^{-bW}G^{ij}W_iW_j \\
  &\leq&C(\psi^{1-\frac{1}{k-l}}+\psi|D\widetilde{W}|+\sum_iG^{ii}),
\end{eqnarray*}
by choosing $b$ sufficiently large.

Next, we consider the function
\begin{equation}\label{Phi}
\Phi:=R\Psi-\widetilde{W},
\end{equation}
where $R$ is a positive constant sufficiently large to be determined later. First choose $\delta\leq\frac{2t}{N}$ such that
\begin{equation}\label{td}
  -td+\frac{N}{2}d^2\leq 0,~\mbox{on}~\partial\omega_{\delta}.
\end{equation}
Then combining with \eqref{vle} and \eqref{td}, we have
$$\Phi\leq 0, ~\mbox{on}~\partial\omega_{\delta}.$$
In order to prove
\begin{equation}\label{ph}
  \Phi\leq 0,~\mbox{on}~\overline{\omega}_{\delta},
\end{equation}
we need to show that $\Phi$ attains its maximum on $\partial\omega_{\delta}$.
 Suppose $\Phi$ attains its maximum at an interior point $x_0\in\omega_{\delta}$. Using (1) of Proposition \ref{sigma}, \eqref{px} and \eqref{lam}, we derive for $k\geq 2$,
\begin{equation}\label{unk}
  \lambda\left(\frac{1}{\omega}\{\gamma^{is}(v_{st}-2\eta_0\delta_{st})\gamma^{jt}\}\right)\in \Gamma_k\subset\widetilde{\Gamma}_k,~~\mbox{on}~~\overline{\omega}_{\delta}.
\end{equation}
Since $\left(\frac{\sigma_k}{\sigma_l}\right)^{\frac{1}{k-l}}$ is concave and homogeneous of degree one, by \eqref{pro} and \eqref{unk}, we get at $x_0$,
\begin{eqnarray*}
  G^{ij}(D^2v-\eta_0I)_{ij}&\geq&(k-l)\psi^{1-\frac{1}{k-l}}G^{\frac{1}{k-l}}(D^2v-\eta_0I,Du)\\
  &\geq& C\psi^{1-\frac{1}{k-l}}\left(\frac{\sigma_k}{\sigma_l}\right)^{\frac{1}{k-l}}(D^2v-\eta_0I)\\
  &\geq& C'\psi^{1-\frac{1}{k-l}},
\end{eqnarray*}
where $C'$ is a constant depending on $k,l,\|u\|_{C^1}$ and $\eta_0$.
Due to $|Dd|=1$ on $\partial\Omega$, we can choose $\delta$ sufficiently small such that
\begin{equation}\label{du}
  \frac{1}{2}\leq |Dd|\leq 1,~\forall ~x\in \omega_{\delta},
\end{equation}
then at $x_0$,
\begin{equation*}
  G^{ij}d_id_j\geq C\sum_iG^{ii}\geq C\psi^{1-\frac{1}{k-l}}.
\end{equation*}
It follows that
\begin{eqnarray*}
  G^{ij}\Psi_{ij} &=&G^{ij}v_{ij}-tG^{ij}d_{ij}+NG^{ij}d_id_j+NdG^{ij}d_{ij} \\
  &\geq&C'\psi^{1-\frac{1}{k-l}}+\eta_0\sum_iG^{ii}+CN\psi^{1-\frac{1}{k-l}}+(Nd-t)G^{ij}d_{ij}\\
  &\geq&CN\psi^{1-\frac{1}{k-l}}+\frac{\eta_0}{2}\sum_iG^{ii},
\end{eqnarray*}
by choosing $t, \delta$ sufficiently small such that $CN\delta+Ct\leq\frac{\eta_0}{2}$.
Note that at $x_0$,
\begin{equation}\label{x0}
  \frac{1}{R}|D\widetilde{W}|=|D\Psi|=|Dv-tDd+NdDd|\leq C(1+t)+C\delta N\leq C,
\end{equation}
by choosing $t, \delta$ sufficiently small.
Therefore at $x_0$,
\begin{eqnarray*}
  0 \geq L\Phi&=& L(R\Psi-\widetilde{W}) \\
  &=& RG^{ij}\Psi_{ij}-R\psi_{u_i}\Psi_i-L\widetilde{W}\\
  &\geq&CNR\psi^{1-\frac{1}{k-l}}+\frac{R\eta_0}{2}\sum_iG^{ii}-CR\psi^{1-\frac{1}{k-l}}|D\Psi|\\
  &&-C\psi^{1-\frac{1}{k-l}}-C\psi|D\widetilde{W}|-C\sum_iG^{ii}\\
  &>&CNR\psi^{1-\frac{1}{k-l}}+\frac{R\eta_0}{4}\sum_iG^{ii}-CR\psi^{1-\frac{1}{k-l}}-CR\psi \\
  &\geq&(CN+C\eta_0-C\psi^{\frac{1}{k-l}}-C)R\psi^{1-\frac{1}{k-l}}>0,
\end{eqnarray*}
 by choosing $N, R$ sufficiently large which is a contradiction.

 Hence the function $\Phi$ can not attain its maximum at an interior point of $\omega_{\delta}$ when $R, N$ are large enough and $\delta, t$ are small enough. Thus \eqref{ph} is proved. Since $\Phi(0)=0$, by Hopf's lemma, we get $\Phi_n(0)\leq 0$. Then
 $$u_{\alpha n}(0)\geq -C.$$
 The above arguments also hold with respect to $-\nabla'_{\alpha}u-\frac{1}{2}\sum_{\beta\leq n-1}u_{\beta}^2$. Hence we get $|u_{\alpha n}(0)|\leq C$.

 $\mathbf{Step~3:}$ Estimates of $u_{nn}(0)$.

Since $H[M_u]>0$, it is sufficient to establish the upper bound of $u_{nn}(0)$. At the origin,
$$u_{\alpha\beta}=-u_n\kappa_{\alpha}^b\delta_{\alpha\beta},~ \mbox{for}~1\leq \alpha,\beta\leq n-1,$$
and
$$g^{ij}=\delta_{ij}-\frac{|Du|^2}{\omega^2}\delta_{in}\delta_{jn}.$$
Hence the matrix of $\{a_{ij}\}$ is
\begin{equation*}
\begin{pmatrix}
-\cfrac{u_n\kappa_1^b}{\omega}& 0 &\cdots& 0 &\cfrac{u_{1n}}{\omega}\\
0 &-\cfrac{u_n\kappa_2^b}{\omega} &\cdots& 0 &\cfrac{u_{2n}}{\omega}\\
\vdots & \vdots &\ddots &\vdots&\vdots\\
0& 0&\cdots&-\cfrac{u_n\kappa_{n-1}^b}{\omega}&\cfrac{u_{n-1 n}}{\omega}\\
\cfrac{u_{n1}}{\omega}&\cfrac{u_{n2}}{\omega}&\cdots&\cfrac{u_{n n-1}}{\omega}&\cfrac{u_{nn}}{\omega^3}
\end{pmatrix}.
\end{equation*}
By Lemma 1.2 of \cite{CNS85} and the estimates of $u_{\alpha\beta}(0), u_{\alpha n}(0)$, there exists a constant $R_1>0$ sufficiently large such that if $u_{nn}>R_1$, then
\begin{equation}
\left\{
\begin{aligned}
&\widetilde{\lambda}_{\alpha}=-\frac{u_n\kappa_{\alpha}^b}{\omega}+o(1),~~\alpha=1,\cdots,n-1,\\
&\widetilde{\lambda}_n = \frac{u_{nn}}{\omega^3}+O(1),
\end{aligned}
\right.
\end{equation}
where $\widetilde{\lambda}(a_{ij})=(\widetilde{\lambda}_1,\cdots,\widetilde{\lambda}_n)$ denotes the eigenvalues of $A=\{a_{ij}\}$.

Since $k<n$, we have at the origin,
\begin{eqnarray*}
  \psi&=& \frac{\sigma_k}{\sigma_l}\left(\sum_i\widetilde{\lambda}_i-\widetilde{\lambda}_1,\cdots,\sum_i\widetilde{\lambda}_i-\widetilde{\lambda}_{n-1},
  \sum_i\widetilde{\lambda}_i-\widetilde{\lambda}_n\right) \\
  &=& \frac{\sigma_k}{\sigma_l}\left(\widetilde{\lambda}_n+\sum_{\alpha=2}^{n-1}\widetilde{\lambda}_{\alpha},\cdots,\widetilde{\lambda}_n+
  \sum_{\alpha=1}^{n-2}\widetilde{\lambda}_{\alpha},\sum_{\alpha=1}^{n-1}\widetilde{\lambda}_{\alpha}\right) \\
  &=&\frac{\sigma_k}{\sigma_l}\left(\widetilde{\lambda}_n-\sum_{\alpha=2}^{n-1}\frac{u_n\kappa_{\alpha}^b}{\omega}+o(1),\cdots,
  \widetilde{\lambda}_n-\sum_{\alpha=1}^{n-2}\frac{u_n\kappa_{\alpha}^b}{\omega}+o(1),-\sum_{\alpha=1}^{n-1}\frac{u_n\kappa_{\alpha}^b}{\omega}
  +o(1)\right)\\
  &\geq&\frac{\widetilde{\lambda}_n^k+o\left(\widetilde{\lambda}_n^{k-1}\right)}{C_n^l\widetilde{\lambda}_n^l+O\left(\widetilde
  {\lambda}_n^{l-1}\right)},
\end{eqnarray*}
by choosing $R_1$ large enough.
It implies the uniform upper bound of $u_{nn}(0)$. Hence \eqref{be} is proved.
\end{proof}

\emph{Proof of Theorem \ref{main}.}
Based on the a priori estimates for non-degenerate case, $C^{2,\alpha}$ estimates can be established by Evans-Krylov theory and higher order estimates followed by Schauder theory. Then the existence result can be derived by the continuity method and the uniqueness assertion is immediate from the maximum principle,  more details see \cite{GT83}. Then Theorem \ref{main} can be proved by approximation as in \cite{JW22}.

\emph{Proof of Theorem \ref{2}.}
We only need to prove interior $C^2$ estimates and mixed boundary estimates, the other estimates are same as Hessian quotient case in Theorem \ref{b}.

$\mathbf{Interior~estimates:}$
Consider the auxiliary function $\overline{Q}=\log H-\log(v-a)$. Applying the similar arguments in Theorem \ref{a} and using \eqref{a1} and  Proposition \ref{27}, it is easy to derive interior estimates.

$\mathbf{Mixed~boundary~estimates:}$
The proof is similar to Step 2 in Theorem \ref{b}, we only give the main ideas here. Consider the function $\Phi$ defined by \eqref{Phi} and assume $\Phi$ attains its maximum at an interior point $x_0\in \omega_{\delta}$.
Let $H$ denote the mean curvature of $M_u$ at $X_0=(x_0, u(x_0))$. As in \cite{JW22}, we consider the following two cases:

$\mathbf{Case~(i)}~H\leq 1$.
Let $\widetilde{a}_{ij}=\omega^{-1}g^{il}u_{lj}$. Then
$$|\widetilde{a}_{ij}(x_0)|\leq \widetilde{C}H\leq\widetilde{C},$$
for each $1\leq i,j \leq n$. Here $\widetilde{C}$ is some uniform positive constant. Then at $x_0$,
\begin{eqnarray*}
  0&=&\Phi_i=(R\Psi-\widetilde{W})_i \\
  &=&R(v_i-td_i+Ndd_i)-be^{-bW}(u_{\alpha i}+\rho_{\alpha}u_{ni}+\rho_{\alpha i}u_n-\sum_{\beta\leq n-1}u_{\beta}u_{\beta i}),
\end{eqnarray*}
for $i=1,\cdots, n$. Hence we derive
\begin{eqnarray*}
  &&\widetilde{a}_{n\alpha}+\rho_{\alpha}\widetilde{a}_{nn}+\omega^{-1}g^{ni}\rho_{i\alpha}u_n-\sum_{\beta\leq n-1}u_{\beta}\widetilde{a}_{n\beta}\\
  &=& \omega^{-1}g^{ni}(u_{i\alpha}+\rho_{\alpha}u_{in}+\rho_{i\alpha}u_n-\sum_{\beta\leq n-1}u_{\beta}u_{i\beta})\\
  &=& Rb^{-1}e^{bW}\omega^{-1}(g^{ni}v_i-tg^{ni}d_i+Ndg^{ni}d_i).
\end{eqnarray*}
Since $v_n(0)=-1$ and $v_{\gamma}(0)=0$ for $\gamma=1, \cdots, n$, we get at $x_0$,
$$\widetilde{a}_{n\alpha}\leq-\widetilde{c}R+C,$$
by choosing $\delta$ and $t$ sufficiently small.
Here $\widetilde{c}$ and $C$ are positive constants depending on $\|u\|_{C^1}$. Then if $R$ is sufficiently large, $\widetilde{a}_{n\alpha}(x_0)$ is unbounded, which implies a contradiction since $|\widetilde{a}_{n\alpha}(x_0)|\leq \widetilde{C}$.

$\mathbf{Case~(ii)}~H>1$.
According to the proof of Lemma \ref{im}, if $\psi^{\frac{1}{k-1}}\in C^1(\overline{\Omega}\times\mathbb{R}\times\mathbb{S}^n)$, we derive
\begin{equation}\label{k1}
  LW\leq C\left(\psi^{1-\frac{1}{k-1}}+\psi|DW|+\sum_iG^{ii}+G^{ij}W_iW_j\right).
\end{equation}
Thus by \eqref{ww}, \eqref{k1} and choosing $b$ sufficiently large,
\begin{equation}\label{11}
  L\widetilde{W}\leq C(\psi^{1-\frac{1}{k-1}}+\psi|D\widetilde{W}|+\sum_i G^{ii}).
\end{equation}

Since $\sigma_k^{\frac{1}{k}}$ is concave and homogeneous of degree one, then by \eqref{pro} and \eqref{unk}, we get at $x_0$,
\begin{equation}\label{nn}
  G^{ij}(D^2v-\eta_0I)_{ij}\geq  C\psi^{1-\frac{1}{k}}\sigma_k^{\frac{1}{k}}(D^2v-\eta_0I)\geq C^{\ast}\psi^{1-\frac{1}{k}},
\end{equation}
where $C^{\ast}$ is a constant depending on $k,\|u\|_{C^1}$ and $\eta_0$.

Combining with \eqref{a0}, \eqref{a1} and \eqref{du}, then at $x_0$,
\begin{equation}\label{gdd}
  G^{ij}d_id_j\geq C\sum_i\sigma_{k-1}(\lambda|i)\geq C\sigma_{k-1}\geq CH^{\frac{1}{k-1}}\psi^{1-\frac{1}{k-1}}\geq C\psi^{1-\frac{1}{k-1}}.
\end{equation}
Using \eqref{psi}, \eqref{nn} and \eqref{gdd}, it follows that at $x_0$,
\begin{eqnarray}\label{gg}
 \nonumber G^{ij}\Psi_{ij}&=&G^{ij}v_{ij}-tG^{ij}d_{ij}+NG^{ij}d_id_j+NdG^{ij}d_{ij} \\
 \nonumber&\geq&C^{\ast}\psi^{1-\frac{1}{k}}+\eta_0\sum_iG^{ii}+CN\psi^{1-\frac{1}{k-1}}+(Nd-t)G^{ij}d_{ij}\\
  &\geq&CN\psi^{1-\frac{1}{k-1}}+\frac{\eta_0}{2}\sum_iG^{ii},
\end{eqnarray}
by choosing $t, \delta$ sufficiently small such that $CN\delta+Ct\leq\frac{\eta_0}{2}$.

Thus using \eqref{Phi}, \eqref{x0}, \eqref{11}, \eqref{gdd}, \eqref{gg} and the condition $\psi^{\frac{1}{k-1}}\in C^{1,1}(\overline{\Omega}\times\mathbb{R}\times\mathbb{S}^n)$, we have at $x_0$,
\begin{eqnarray*}
  0 \geq L\Phi&=& RG^{ij}\Psi_{ij}-R\psi_{u_i}\Psi_i-L\widetilde{W}\\
  &\geq&CNR\psi^{1-\frac{1}{k-1}}+\frac{R\eta_0}{2}\sum_iG^{ii}-CR\psi^{1-\frac{1}{k-1}}|D\Psi|\\
  &&-C\psi^{1-\frac{1}{k-1}}-C\psi|D\widetilde{W}|-C\sum_iG^{ii}\\
  &>&CNR\psi^{1-\frac{1}{k-1}}+\frac{R\eta_0}{4}\sum_iG^{ii}-CR\psi^{1-\frac{1}{k-1}}-CR\psi \\
  &\geq&(CN+C\eta_0-C\psi^{\frac{1}{k-1}}-C)R\psi^{1-\frac{1}{k-1}}>0,
\end{eqnarray*}
 by choosing $N, R$ sufficiently large which is a contradiction. Hence $\Phi$ only attains its maximum on $\partial\omega_{\delta}$ and mixed boundary estimates are followed by Hopf's lemma.


\end{document}